\documentclass{amsart}
\title{Generic Large Cardinals as Axioms}
\author{Monroe Eskew}
\address{Kurt G\"odel Research Center \\
University of Vienna \\
W\"ahringer Strasse 25 \\
1090 Wien \\
Austria \\
monroe.eskew@univie.ac.at.}

\date{}

\thanks{The author would like to thank Sean Walsh and Neil Barton for many helpful discussions.}

\usepackage{color} 
\usepackage{amssymb} 
\usepackage{amsmath} 
\usepackage[utf8]{inputenc} 
\usepackage{enumerate}
\usepackage{amsthm}
\usepackage{cite}

\newtheorem{theorem}{Theorem}
\newtheorem{lemma}[theorem]{Lemma}

\newtheorem{corollary}[theorem]{Corollary}

\newtheorem*{definition}{Definition}

\DeclareMathOperator{\cf}{cf}

\DeclareMathOperator{\col}{Col}

\DeclareMathOperator{\ns}{NS}

\newcommand{\p}{\mathcal{P}}

\begin{document}
\maketitle


\begin{abstract}
We argue against Foreman's proposal to settle the continuum hypothesis and other classical independent questions via the adoption of generic large cardinal axioms.
\end{abstract}

Shortly after proving that the set of all real numbers has a strictly larger cardinality than the set of integers,
Cantor conjectured his Continuum Hypothesis (CH):
that there is no set of a size strictly in between that of the integers and the real numbers \cite{cantor}.  A resolution of CH was the
first problem on Hilbert's famous list presented in 1900 \cite{hilbert}.  G\"odel made a major advance by constructing a model of the Zermelo-Frankel (ZF) axioms for set theory
in which the Axiom of Choice and CH both hold, starting from a model of ZF.  This showed that the axiom system ZF, if consistent on its own, could not disprove Choice, and that ZF
with Choice (ZFC), a system which suffices to formalize the methods of
ordinary mathematics, could not disprove CH \cite{godelL}.  It remained unknown at the time whether models of ZFC could be found in which CH was false, but G\"odel began
to suspect that this was possible, and hence that CH could not be settled on the basis of the normal methods of mathematics.
G\"odel remained hopeful, however, that new mathematical axioms known as ``large cardinals'' might be
able to give a definitive answer on CH \cite{godel}.

The independence of CH from ZFC was finally solved by Cohen's invention of the method of forcing \cite{cohen}.
Cohen's method showed that ZFC could not prove CH either, and in fact 
could not put any kind of bound on the possible number of cardinals between the sizes of the integers and the reals.  L\'evy and Solovay further developed
the forcing machinery, and noticed that it also destroyed G\"odel's hopes for large cardinals.  Forcing allowed one to manipulate the cardinal value of the set 
of reals, passing from one model of ZFC to another giving a different answer on CH, without disturbing any large cardinals in the process \cite{levy}.

This was not the last word on CH from the community of set theorists.  Several programs to develop acceptable axioms that settle CH have been put forward.  Matthew Foreman, a leading set theorist, has suggested a solution to CH via axioms called ``generic large cardinals.''
Our goal here is to critically examine Foreman's proposal.  First, we describe the goals these axioms are supposed to meet and the kinds of considerations in their favor, highlighting the claim that the favorable considerations for traditional large cardinals transfer to the generic ones.  Second, we discuss many technical difficulties in accommodating generic large cardinals in a single axiomatic framework, and present some new ``mutual inconsistency results'' that raise troubles for the program.  Third, we examine the considerations in favor of traditional large cardinals and argue that they do not have the same import for the generic variety.  Finally, we consider an alternative take on these kinds of axioms that seems to avoid the technical difficulties, but sacrifices some of the original goals of the program.

\section{Foreman's program}
In accord with G\"odel \cite{godel}, Foreman regards CH as a legitimate and important mathematical problem, despite its independence from ZFC.  
In Foreman's view, such problems may be resolved by the adoption of new axioms, a method which is in line with the historical practice of mathematics, as 
``whatever process led to the acceptance of ZFC as an axiomatization of mathematics (despite its controversial beginnings) may lead to other assumptions that
settle, or partially settle most of the problems we are interested in'' \cite{GLC}.

Of course, the everyday mathematical work of performing calculations and proving theorems is not the same kind of thing as making a commitment to a new axiom.
So how does one assess axiom candidates?  Foreman lays out two categories of evidence, which he calls ``primary considerations'' and ``secondary considerations'' \cite{foremanphil}.  Roughly, primary considerations have
to do with the conceptual content of the axioms, while secondary considerations have to do with various utilitarian features of the theories they generate.
Included among these features is \emph{completness}, the effectiveness of the axiom system at answering questions.  We will not dispute the claim that generic large cardinals fare particularly well with respect to many secondary considerations.  Indeed, they are able to settle several classical questions about relatively small cardinals, including CH, and they have many of the same implications for descriptive set theory as conventional large cardinals.\footnote{See \cite{foremanhandbook}.}

What Foreman describes as primary considerations seem to be closely linked to ``intrinsic justifications,''\footnote{See \cite{maddy}.} but perhaps also include considerations of ``naturalness'' that may not count as justifications in the sense of valid arguments for conclusions.  G\"odel suggests something along these lines, saying that large cardinals are ``axioms which are only the natural continuation of the series of those set up so far" \cite{godel}.  
Foreman argues that the primary considerations for traditional large cardinals transfer to the generic large cardinals.  Since many set theorists accept large cardinals as intrinsically justified or natural, the implication is that they should therefore view generic large cardinals in the same way.  Our aim in this paper is to argue against this transference claim, without attempting a thorough analysis of the underlying notion of ``primary consideration.''  We will first argue that the transference claim leads to unwelcome consequences, and then critically examine whether some of the well-known primary considerations for traditional large cardinals do indeed apply equally to generic ones.


Foreman describes large cardinals as ``a successful axiom system''  \cite{foremanphil}.  With regard to primary considerations in their favor, Foreman says little in \cite{foremanphil} beyond 
noting the ``sociological fact that the dominant view among those actively searching for true axioms that extend ZFC is that Large Cardinal axioms are true.''  This is reminiscent of an argument for the Axiom of Choice made by Zermelo, who noted that ``it is applied without hesitation everywhere in mathematical deduction" \cite{zermelo}.
In other words, specialists in the relevant fields demonstrate in their mode of working that the principles in question are natural or intuitive. 
Foreman says the generic large cardinals are ``straightforward generalizations of conventional large cardinals," and that the ``evidence for large cardinals, when suitably viewed, does not distinguish between conventional large cardinals and generic large cardinals'' \cite{foremanphil}.  Indeed, Foreman gives a characterization of generic large cardinals that includes many conventional large cardinals as special cases.

A key mathematical difference between the conventional and generic large cardinals is the \emph{mutual inconsistency phenomenon} observed with the generic variety.  There are pairs of generic large cardinal axioms that contradict one another. In some cases, both members of the pair are known to be consistent relative to conventional large cardinals.  Since the conventional large cardinals form a directed order in strength, such phenomena are not possible for the conventional large cardinals that have been studied.  If a pair of them were to be inconsistent with each other, then a single large cardinal axiom would be inconsistent by itself.  This is because for any finite collection $S$ of large cardinal axioms, one only needs to look a bit further up the hierarchy to find an axiom which implies the existence of a model satisfying all statements in $S$.\footnote{It should be noted that this is not a meta-theorem about an abstract notion of large cardinals, but more of an empirical fact about the axioms that have been studied.  Indeed, there is no generally accepted formal definition of what a large cardinal is, though there is certainly enough resemblance and coherence between the extant axioms to warrant the use of a general term.}

What are we to make of this situation?  If we are seeking to adopt a new axiom system to resolve independent questions, we certainly should not adopt an inconsistent one.  According to Foreman \cite{foremanhandbook}, ``While the counterexample to mutual consistency is certainly very troubling, it may not be fatal to the program of looking to generalized large cardinals for true extensions of ZFC.''  He states that ``the `mutual inconsistency phenomenon' seems rare," and that the important question is, ``which generalized large cardinals are true.''  In other words, he predicts that a deeper understanding of the situation will yield reasons to accept some generic large cardinals as axioms and reject others.


A straightforward acceptance of the transference claim immediately leads us into problems.  For if everything worthy of the name ``generic large cardinal" is as deserving of axiomhood as conventional large cardinals, 
then we have several equally deserving axiom candidates, which cannot be adopted simultaneously.  We cannot accept a simple principle such as ``generic large cardinal axioms are true,'' and we seem to be left with no alternative principle that favors some generic large cardinals over others.
But perhaps the primary considerations in favor of conventional large cardinals do not legitimately apply to some hypotheses we currently call ``generic large cardinals."
If the mutual inconsistency phenomenon is indeed quite rare, then it becomes plausible that a more refined understanding of the considerations in favor of conventional and generic large cardinals may resolve the problem by guiding us to a choice between conflicting axioms, perhaps by informing a scheme for what counts as a generic large cardinal axiom that excludes problematic cases.
This could be guided by new mathematical information, as Kunen's proof of the nonexistence of ``Reinhardt cardinals'' led to a maturation of the large cardinal theory rather than its collapse \cite{srk}.
However, we aim to show in the next section, by collecting some known results and presenting some new ones, that the phenomenon may be unsettlingly common, and hence the prospects for the program may be bleak.

\section{Inconsistencies}
Without attempting a comprehensive definition, Foreman in \cite{GLC} characterizes generic large cardinals as axioms that ``assert the existence of an elementary embedding $j : V \to M$, where $j$ is definable in a forcing extension of the universe $V[G]$.''  ($M$ is assumed to be a transitive class definable in $V[G]$.)  Conventional large cardinals (at least those of sufficiently high strength) fit the same description, except that they do not allow $j$ to be generated by forcing.

With no restrictions on the hypotheses of this kind that one may consider as axiom candidates, we quickly run into inconsistencies.  For example, there is a relationship between the kind of forcing used to generate the embedding $j$ and the possible critical point of $j$.  The kind of forcing employed can dictate the allowable critical points, and a stipulation of a certain critical point can restrict the class of employable forcings.  If we stipulate the critical point to be a successor cardinal, then this cardinal must be collapsed by the forcing, since if $\kappa = \mu^+$ is the least ordinal moved by $j : V \to M$, then $M \models \mu < \kappa < j(\kappa) = \mu^+.$  If we stipulate the forcing to have the countable chain condition, then the critical point must be weakly inaccessible and at most the value of the continuum.\footnote{See \cite{jech} or  \cite{kanamori}.}
Since some generic large cardinals imply the continuum is not so large, we immediately get some mutual inconsistencies.

A less obvious restriction on generic embeddings occurs near singular cardinals.  Since any forcing generating an elementary embedding with critical point a successor cardinal $\kappa = \mu^+$ must collapse $\kappa$ to have the same cardinality as $\mu$, the best chain condition one can get for the forcing is the $\kappa^+$-c.c., and thus the minimal possible size of the forcing is $\kappa$.  The former property is equivalent to the existence of a $\kappa^+$-saturated normal ideal on $\kappa$, and the latter a $\kappa$-dense normal ideal on $\kappa$.  These objects are the combinatorial surrogates for the kinds of generic embeddings they generate, in analogy to the countably complete ultrafilters on various sets associated to various kinds of conventional large cardinal embeddings.  These objects are often referred to by the abridged names, ``saturated ideal'' and ``dense ideal'' on $\kappa$.

A dense ideal on $\omega_1$ is known to be equiconsistent with infinitely many Woodin cardinals \cite{woodin-book}, and for a general successor of a regular cardinal, a model can be obtained from an almost-huge cardinal \cite{dense}.  Foreman \cite{foremanmore} obtained a model of ZFC in which \emph{every} successor cardinal carries a saturated ideal, and therefore a natural question is whether one can obtain the stronger property of a dense ideal on $\kappa$, where $\kappa$ is the successor a singular cardinal.  In \cite{dense}, the author showed that this contradicts the generalized continuum hypothesis (GCH).  Using a lesser-known forcing lemma, we easily obtain a contradiction without GCH:

\begin{lemma}[Sakai \cite{sakaisemiproper}]
If $\mathbb P$ is a partial order of size $\kappa$, where $\kappa$ is a regular cardinal, then $\mathbb P$ forces $\cf(\kappa) = | \kappa |$.
\end{lemma}

\begin{corollary}
If $\kappa$ is the successor of a singular cardinal, then there is no dense ideal on $\kappa$.
\end{corollary}

\begin{proof}
Suppose $\mu$ is singular, $\kappa = \mu^+$, and there is a dense ideal on $\kappa$.
 Let $j : V \to M \subseteq V[G]$ be a generic embedding arising from forcing with the ideal.  The critical point of $j$ is $\kappa$ and $M^\kappa \cap V[G] \subseteq M$.  By elementarity and the closure of $M$, $|\kappa | = \mu$ and $\cf(\kappa) < \mu$ in $V[G]$.  This contradicts Sakai's Lemma since $G$ is generic for a forcing of size $\kappa$.
 \end{proof}

Deeper restrictions on generic elementary embeddings come from work regarding the canonical \emph{nonstationary ideal} on a regular cardinal $\kappa$, denoted $\ns_\kappa$.
Work of Gitik and Shelah \cite{gsless} shows that is not possible for $\ns_\kappa$ to be saturated unless $\kappa = \aleph_1$, which is consistent relative to a Woodin cardinal \cite{woodin-book}.  For various restricted versions of this, the relative consistency is open--for example the statement that $\ns_{\aleph_2}$ restricted to ordinals of cofinality $\aleph_1$ is $\aleph_3$-saturated.  However, Foreman and Magidor \cite{fmcx} proved that if there are sufficient conventional large cardinals in the background, then the above statement about $\ns_{\aleph_2}$ contradicts the $\aleph_2$-saturation of $\ns_{\aleph_1}$.  In particular, these hypotheses give different answers regarding whether a counterexample to CH exists in the inner model $L(\mathbb R)$.

With regard to the inconsistencies involving the nonstationary ideal, Foreman has a response.  He characterizes the $\aleph_2$-saturation of $\ns_{\aleph_1}$ as ``an anomaly that does not fit the general situation'' \cite{foremanhandbook}, suggesting that a clearer picture of the mathematical situation may lead to a refined formulation of the notion of ``generic large cardinal'' that does not allow this contradiction to go through.  This leads Foreman to an ``informal working definition'' of a generic large cardinal hypothesis as an assertion of the forcability of an elementary embedding $j : N \to M$ for transitive structures $N$ and $M$, with three paramters:
\begin{enumerate}
 \item Where $j$ sends the ordinals.
 \item How big $N$ and $M$ are.
 \item The nature of the forcing.
\end{enumerate}
As he states, ``This mechanism appears to define away the anomaly of $\ns_{\omega_1}$,'' since there is no place in this framework for asserting that the generic embedding is related to a particular definable ideal.  He clarifies that by ``the nature of the forcing,'' he means its isomorphism type as a boolean algebra.  The assertion, ``such-and-such isomorphism type can be represented as such-and-such,'' is not a permitted part of a generic large cardinal hypothesis.

This informal defintion does not have enough information to exclude inconsistent cases, as Foreman notes that ``one cannot adjust these parameters arbitrarily.''  As mentioned above, there is an interplay between axes (1) and (3).  There are obvious interplays between (1) and (2), as making $N$ and $M$ ``larger'' tends to introduce more absoluteness and thus restrict the action of $j$.
There are also limitations regarding (2) analogous to Kunen's inconsistency.\footnote{See \cite{foremanhandbook}, Section 6.2.}
Nonetheless, one might reasonably hope that we can specify some consistent general framework that captures most of the instances of generic large cardinals that have been studied.

Whatever underlying concept motivates Foreman's working definition, it does appear that there are instances of generic large cardinals falling under the concept that are mutually contradictory.  Foreman mentions two examples that he finds troubling.  The first example shows that the following ``particularly attractive sounding axiom'' asserting the existence of a family of ideals is false:
``For all $\aleph_2$-c.c.\ Boolean algebras $\mathbb B$ of cardinality less than or equal to $2^{\omega_2}$ that collapse $\omega_1$, there is a normal fine ideal $I$ on $[\omega_2]^{\omega_1}$ such that $\p([\omega_2]^{\omega_1})/I \cong \mathbb B$.''

If we substitute $\mathbb B = \col(\omega,\omega_1)$, then we can prove CH,\footnote{See \cite{foremanhandbook}, Section 5.3.} and that there is also an ideal $J$ on $\omega_1$ such that $\p(\omega_1)/J \cong \col(\omega,\omega_1)$.  These two hypotheses together prove a certain partition property $P$.\footnote{$P$ asserts that whenever we have a coloring of the rectangle $\omega_2 \times \omega_1$ in countably many colors, there are sets $A \subseteq \omega_2$ and $B \subseteq \omega_1$, both uncountable, such that the coloring is constant on $A \times B$.}
If we substitute $\mathbb B = \col(\omega,{<}\omega_2)$, then from this hypothesis we can prove $\neg P$.  Although the existence of \emph{any} $\aleph_2$-saturated normal ideal on $[\omega_2]^{\omega_1}$ is not known to be consistent relative to large cardinals, Foreman notes that the same argument can be carried out on the basis of the conjunction of CH with two weaker ideal hypotheses known to be individually consistent (with GCH) relative to conventional large cardinals.\footnote{These are a dense ideal on $\omega_1$ and a normal ideal on $[\lambda]^{\omega_1}$ with associated forcing $\col(\omega,{<}\lambda)$, for some inaccessible $\lambda$.}

In \cite{rims}, the author generalized this result as follows, allowing us to drop the hypothesis of CH from the original case:
\begin{theorem}Suppose $\kappa$ is a successor cardinal, there is a $\kappa$-complete, $\kappa$-dense ideal on $\kappa$, and $\lambda > \kappa$ is a cardinal. 
 \begin{enumerate}
  \item If $\lambda$ is a successor cardinal, there is no normal ideal on $[\lambda]^\kappa$ whose associated forcing is $\lambda$-absolutely $\lambda$-c.c.\ and uniformly $\lambda$-dense.
  \item If $\lambda$ is a limit cardinal, then there is no $\lambda^+$-saturated normal ideal on $[\lambda]^\kappa$.
 \end{enumerate}
\end{theorem}
Although we will not repeat here the definitions of $\lambda$-absolutely $\lambda$-c.c.\ and uniformly $\lambda$-dense, we note the conjunction of these properties covers a broad class of partial orders that includes $\col(\mu,{<}\lambda)$ when $\lambda$ is regular and $\lambda^{<\mu} = \lambda$.
If a proponent of generic large cardinals as axioms is troubled by the more specific inconsistency, then certainly this generalization should be all the more disconcerting, as so many more pairs of generic large cardinals are ruled out.  In particular, dense ideals on successor cardinals are inconsistent with any reasonably saturated ideal which collapses an inaccessible to be that same successor.  As noted in \cite{rims}, stronger properties of ideals that generically map successor to inaccessible cardinals can rule out even saturated ideals on those successor cardinals.

After discussing the above mutual inconsistency phenomenon in \cite{foremanhandbook}, Foreman defined a particularly strong notion of ``minimally generically $n$-huge'' as an example of a type of generic large cardinal axiom:

\begin{definition}
For finite $n >0$,  a cardinal $\kappa = \mu^+$ is \emph{minimally generically $n$-huge} iff there
is a normal, fine, $\kappa$-complete ideal $I$ on $Z = [\kappa^{+n}]^{\kappa^{+n-1}}$ such that $\p(Z)/I$ has a dense
set isomorphic to $\col(\mu,\kappa)$.
\end{definition}

If we additionally assume that $\kappa^{<\mu} \leq \kappa^{+n}$,
then whenever $G \subseteq \p(Z)/I$ is generic, there is an elementary embedding $j : V \to M \subseteq V[G]$ such that
$M$ is closed under $\kappa^{+n}$-sequences from $V[G]$, and
$j(\kappa^{+m}) = \kappa^{+m+1}$ for all $m<n$.
Since $\col(\mu,\kappa)$ actually collapses $\kappa^{<\mu}$ to $\mu$, and $(\kappa^+)^V$ is a cardinal in $V[G]$,
we must have in this context that $\kappa^{<\mu} = \kappa$, meaning that the forcing associated to the ideal is of minimal possible density,
and is in fact uniquely characterized as the $\mu$-closed partial order of this density with these collapsing effects.\footnote{See \cite{cummingshandbook}, Section 14.}   Furthermore, a deeper theorem of Foreman
shows that in the case $\mu = \omega$, the hypothesis implies CH and in fact GCH up to $\kappa^{+n}$.\footnote{See \cite{foremanhandbook}, Theorem 5.9.}  
Using a simple trick, we can draw the same conclusion
for larger values of $\mu$.\footnote{See \cite{dense}, Corollary 3.4.  If $\kappa = \mu^+$ and there is a $\kappa$-complete, normal, $\kappa$-dense ideal on $Z$ but $2^\mu > \kappa$, then we can collapse $\mu$ to $\omega$ and preserve all of these properties, getting a model that is ruled out by Foreman's Theorem.}

Along these lines, it is natural to define a cardinal $\kappa = \mu^+$ to be \emph{minimally generically almost-huge} if there is a $\kappa$-complete ideal $I$ on $\kappa$ such that $\p(\kappa)/I$ has a dense set isomorphic to $\col(\mu,\kappa)$.
A standard projection argument shows that if $\kappa = \mu^+$, $\kappa^{<\mu} = \kappa$, $\kappa$ is minimally generically $n$-huge, and $0 < m \leq n$, then $\kappa$ is minimally generically $m$-huge and minimally generically almost-huge.

As \cite{foremanhandbook} went to press, Woodin showed in an unpublished note that it is inconsistent for $\omega_1$ to be minimally generically 3-huge while $\omega_3$ is minimally generically 1-huge.  Woodin's argument was somewhat specific to the cardinals involved.  We prove below the following generalization of his result using a different argument.


\begin{theorem}
 \label{genhuge}
Suppose $\kappa = \mu^+$ and $n \geq 2$.  Then the following are mutually inconsistent:
\begin{enumerate}
 \item $\kappa$ is minimally generically $n$-huge.
 \item For some $m$, $0 < m < n$, there is an ideal $I$ on $\kappa^{+m}$ such that $\p(\kappa^{+m})/I$ is forcing-equivalent to a $\kappa$-closed partial order.\footnote{With minor modifications to the arguments, ``$\kappa$-closed'' can be weakened to ``$\kappa$-strategically-closed.''}
\end{enumerate}
\end{theorem}

The second hypothesis includes the case that $\kappa^{+m}$ is minimally generically almost-huge, but is much weaker, and in fact equiconsistent with a measurable cardinal.\footnote{For example, if we force with $\col(\mu,{<}\kappa)$, where $\kappa$ is measurable, then the ideal $I$ generated by the dual of a normal ultrafilter $\mathcal U$ on $\kappa$ has the property that $\p(\kappa)/I$ has a dense set isomorphic to $\col(\mu,{<}j_{\mathcal U}(\kappa))$.  For the reverse direction, such an ideal is \emph{precipitous}, which implies that $\kappa$ is measurable in an inner model.}
The second hypothesis also implies that $(\kappa^{+m-1})^\mu = \kappa^{+m-1}$, since otherwise, a generic embedding arising from $I$ would stretch an enumeration of $\p_\kappa(\kappa^{+m-1})$ to reveal new elements in the generic extension.  But a $\kappa$-closed partial order does not add any ${<}\kappa$-sequences of ordinals.  \emph{A fortiori}, $\kappa^{<\mu} \leq \kappa^{+n}$, and so by the prior remarks, the two hypotheses together imply GCH holds on the interval $[\mu,\kappa^{+n}]$.

The theorem shows that degrees of minimal generic hugeness on nearby successor cardinals contradict one another.
It will follow from a more general lemma that is a bit clumsier to state.  Before beginning the argument, we gather some notions and facts we will need:

\begin{definition}[Hamkins \cite{hamkins}]
 A partial order $\mathbb P$ has the \emph{$\kappa$-approximation property} when for all $X \in V$ and all $\mathbb P$-names $\dot{Y}$, 
 if it is forced that $\dot Y \cap x \in V$ for all $x \in \p_\kappa(X)^V$, then it is forced that $\dot Y \cap X \in V$.
\end{definition}

\begin{definition}[Usuba \cite{usubaapprox}]
 A partial order has the \emph{strong $\kappa$-c.c.}\ if it has the $\kappa$-c.c.\ and forcing with it adds no cofinal branches to $\kappa$-Suslin trees.
\end{definition}

We note two ways of guaranteeing that $\mathbb P$ has the strong $\kappa$-c.c.: $\mathbb P$ has the $\mu$-c.c. for some $\mu < \kappa$, or $\mathbb P \times \mathbb P$ is $\kappa$-c.c.  The latter follows easily if $|\mathbb P| < \kappa$.
The following result of Usuba \cite{usubaapprox} improves results of Hamkins, Mitchell, and Unger:

\begin{theorem}[Usuba]
 Suppose $\kappa$ is a regular cardinal, $\mathbb P$ is a nontrivial $\kappa$-c.c.\ partial order, and $\dot{\mathbb Q}$ is a $\mathbb P$-name for a $\kappa$-closed partial order.  Then $\mathbb P * \dot{\mathbb Q}$ has the $\kappa$-approximation property if and only if $\mathbb P$ has the strong $\kappa$-c.c.
\end{theorem}

The next lemma implies Theorem \ref{genhuge}.  As noted above, the hypotheses of Theorem \ref{genhuge} imply GCH on the interval $[\mu,\kappa^{+n}]$, and thus the hypotheses of the next lemma are satisfied with $\mathbb P = \col(\mu,\kappa)$, $\kappa = \kappa_0$, $\kappa_1= \kappa^+$, $\lambda_0 = \kappa^{+m-1}$, and $\lambda_1= \kappa^{+m}$.

\begin{lemma}
 Suppose there are regular cardinals $\kappa_0,\kappa_1,\lambda_0,\lambda_1$ such that $\lambda_0^{<\lambda_0} = \lambda_0$ and $\kappa_0 \leq \lambda_0$.
 Suppose $\mathbb P$ is a strongly $\kappa_1$-c.c.\ nontrivial partial order such that whenever $G \subseteq \mathbb P$ is generic, there is an elementary embedding $j : V \to M \subseteq V[G]$ with:
 \begin{enumerate}
  \item $j(\kappa_0) = \kappa_1$.
    \item $j(\lambda_0) = \lambda_1$.
  \item $M$ is closed under $\lambda_1^+$-sequences from $V[G]$.
 \end{enumerate}
 Then there is no $\lambda_0^+$-complete ideal $I$ on $\lambda_0^+$ whose associated forcing is equivalent to a $\kappa_0$-closed partial order.\footnote{If $\kappa_0 = \lambda_0$, we may drop the cardinal arithmetic assumption, as the existence of a $\kappa_0^+$-complete ideal on $\kappa_0^+$ whose associated forcing is equivalent to a $\kappa_0$-closed partial order implies $\kappa_0^{<\kappa_0} = \kappa_0$.}

\end{lemma}

\begin{proof}
Suppose $j : V \to M \subseteq V[G]$ is a generic elementary embedding as hypothesized.
By the closure of $M$ and the chain condition, $\lambda_1$ is regular in $V[G]$, and $j(\lambda_0^+) = (\lambda_1^+)^{V[G]} = (\lambda_1^+)^V$.

Since $V \models \lambda_0^{<\lambda_0} = \lambda_0$, it follows that $V \models \lambda_1^{<\lambda_1} = \lambda_1$.  For otherwise, in $V$ there is a surjection $f : [\lambda_1]^{<\lambda_1} \to \lambda_1^+$, which would be in $M$ by the closure of $M$.  But by elementarity, $M \models \lambda_1^{<\lambda_1} = \lambda_1$.  Therefore, then by the well-known result of Specker \cite{specker}, there is a $\lambda_1^+$-Aronszajn tree $T \in V$.

Suppose $I$ is as hypothesized.  $M \models j(I)$ is a $\lambda_1^+$-complete ideal on $\lambda_1^+$ whose associated forcing is equivalent to a $\kappa_1$-closed partial order.
By the closure of $M$, this is true in $V[G]$ as well.  If we force with $\p(\lambda_1^+)/j(I)$ over $V[G]$, then we get an elementary embedding
$i : V[G] \to N$ with critical point $\lambda_1^+$.  We have $i(T) \restriction (\lambda_1^+)^V = T$, and a branch through $T$ can be found by looking below a
node in $i(T)$ at level $(\lambda_1^+)^V$.  Therefore, by forcing with $\mathbb P * \p(\lambda_1^+)/j(I)$, we obtain a new set $b \subset T$ such that for every
$x \in \p_{\kappa_1}(T)^V$, $b \cap x \in V$.  This is a failure of $\kappa_1$-approximation, in contradiction to Usuba's Theorem.
\end{proof}

In summary, many pairs of seemingly natural candidates for generic large cardinal axioms turn out to contradict one another.  A fixed successor cardinal cannot be generically large in certain different ways at the same time, where the nature of the forcing producing the generic embedding varies.  A canonical notion of minimal generic $n$-hugeness cannot hold simultaneously on nearby successor cardinals.  The pervasiveness of such examples, combined with the absence of any apparent salient difference between conflicting hypotheses, lowers hopes for a single consistent template for such axioms.  We are thus compelled to scrutinize the possible primary considerations in favor of these principles.

\section{Weight of the evidence}
Do the primary considerations for conventional large cardinals really apply equally to generic large cardinals?  
The key difference between conventional and generic large cardinals is the admissibility of forcing to construct a nontrivial elementary embedding of $V$ into a transitive class.  On its face, this involves the introduction of an object $G \notin V$ to form a larger model $V[G]$ in which the embedding can be defined.  If the axioms of set theory are meant to describe what is true in \emph{the} universe of sets $V$, this seems curious.  Generic large cardinals would seem to be principles about local regions of a set-theoretic multiverse.  Indeed, Foreman states in \cite{GLC} that a generic large cardinal hypothesis ``allows one to state `symmetry principles' that can hold in a generic extension of the universe."  Foreman makes no attempt to dress the motivating picture for generic large cardinals in a formalism that avoids reference to extensions of the universe, other than to say that the forcability of generic embeddings is equivalent to the existence of ideals in $V$ with certain properties.  But it is clearly the formulation in terms of generic embeddings, rather than combinatorial properties of ideals, that is the source of motivation for these principles.  Putting aside for the moment general worries about the relationship of generic embeddings to multiversism, let us examine several of the primary considerations for large cardinals, and assess whether they apply to the generic variety.

\subsection{Arguments from authority}
In \cite{foremanphil}, Foreman largely avoids stating what primary considerations in favor of large cardinals there are, opting instead to argue simply \emph{that they exist.}  He cites the fact that many expert set theorists treat large cardinal assumptions ``\emph{as if} they were true,'' which he says shows they ``have intuitions about the veridicality of the axioms."  Moreover, these intuitions are ``educated," and may be likened to the refined judgements of professional wine tasters, in the sense that one has reason to trust the experts' opinions even if it is difficult for non-experts to inspect the basis for the judgments.\footnote{Of course, it is somewhat controversial that such connoisseurs are really tracking some objective facts.}

However, this kind of evidentiary picture is not so similar with generic large cardinals.  Foreman acknowledges as much, noting that ``the generic large cardinals are much less studied than conventional large cardinals, and so it is hard to supply the same kind of historical or sociological evidence for the intuitive content of the large cardinals.''  We may also add that there are not many set theorists who appear to regard generic large cardinal principles as axiomatically true.  Foreman's aim in his several philosophical writings about generic large cardinals is to prescribe these as axioms to an audience that includes the community of set theorists, rather than to describe a existing consensus.  Even if a consensus among experts about propositions in their field of expertise counts as evidence in favor of those propositions, this kind of evidence is presently lacking for generic large cardinals.

\subsection{Generalization}
One motivating idea for many conventional large cardinals, discussed for example in \cite{srk}, is that the many properties of $\omega$ that are the result of the gap between the finite and infinite should generalize to higher infinities, displaying a similarly vast difference in size between varieties of infinite sets.  Insofar as such considerations are about sheer relative size, they do not seem to apply to generic large cardinals, where the focus is on \emph{accessible} cardinals such as the $\aleph_n$ for finite $n$.  On the other hand, generic large cardinals do indeed generalize the previously-studied large cardinals by introducing the possibility of elementary embeddings generated by forcing.  There does not appear to be a \emph{justification} lurking here as to why these kinds of generalizations should be true, but we grant that they may appear \emph{natural}.  However, the feeling of naturalness may be mitigated in the face of other compelling facts, such as in the case of Reinhardt cardinals.  We believe that the mathematical facts laid out in the previous section play a similar role.  As Foreman points out in \cite{foremanphil}, evolutions of views of what is natural in mathematics occur quite often, such as with the discovery of continuous but nowhere differentiable functions or sets of reals that are not Lebesgue-measurable.

\subsection{Reflection and resemblance}
In \cite{GLC}, Foreman asserts that a hypothesis that a small cardinal is generically large ``allows these cardinals to have similar reflection and resemblance properties as posited by large cardinal axioms on highly inaccessible cardinals."  These ideas appear in \cite{srk} as motivating principles for conventional large cardinals.  The idea of \emph{reflection} is that the mathematical universe is vast enough that it cannot be characterized by properties that hold in it, and instead these properties must already be satisfied by robust set-sized approximations to it, such as rank-initial segments $V_\alpha$.  Furthermore, this phenomenon should itself reflect, so that many $V_\alpha$ have their properties reflected by smaller $V_\beta$.  The idea of \emph{resemblance} is that, for the same kinds of reasons, many levels of the cumulative hierarchy, or many members of various other classes, should resemble one another, perhaps via mechanisms such as elementary embeddings.  An example of a precise implementation of this idea is Vop\v enka's Principle, which says that for any proper class of structures in the same language, there exists an elementary embedding from one member of the class into another.

Many kinds of generic large cardinals entail the reflection of various properties, and the reflection that one gets from a generic embedding is indeed a key component of many arguments involving these objects, as many examples in \cite{foremanhandbook} show.  Also, the embeddings themselves may be viewed as kinds of resemblance properties.  However, there are limitations.  The $\Pi_1$ statement that $\kappa$ is a cardinal fails to reflect on a final segment of $\alpha < \kappa$ if $\kappa$ is a successor cardinal.  Furthermore, there is a very simple sense in which the $\aleph_n$ for finite $n$ do \emph{not} resemble one another; each has a relatively simple definition.

Thus, the typical cardinals one considers as generically large cannot enjoy the kind of full-fledged reflection and resemblance that is possible at conventional large cardinals.  These observations can be seen as manifestations of the fact that the ``symmetry'' which appears via a generic embedding does not occur in $V$ but in an outer model.  The idea that the mathematical universe has structures too rich to be pinned down by such-and-such kinds of properties does not seem to motivate the generic largeness of small cardinals.  There \emph{are} plenty of resources for describing low levels of the cumulative hierarchy.  The resemblance between low-rank objects exhibited by generic a embedding only appears by changing the background universe and thus changing the properties of some objects.  While a principle asserting the occurrence of this kind of phenomenon may be well-motivated, it is not motivated by \emph{the same} ideas that are commonly put forward for conventional large cardinals.

\section{Multiversism}
The motivating picture for generic large cardinals is ostensibly about the relationship of $V$ to a generic extension $V[G]$.  This is in itself a big difference with conventional large cardinals, which are all unambiguously about one universe $V$.  We would like to suggest a way of saving generic large cardinals as axiomatic principles of a sort, in a way that embraces this difference.  The cost, however, is that we give up on using generic large cardinals to arrive at final answers to classical independent questions like CH.

Perhaps, to the working set theorist, generic embeddings have some intuitive appeal as first principles.  But because of mutual inconsistencies, it is difficult if not impossible to treat them as axiom candidates in the normal sense.  Given that they have their most appealing formulation in terms of a relationship between several models of set theory, it may be in their own spirit to state them in a way that allows the \emph{domain} of the elementary embedding to vary.  (To be fair, this possibility is already hinted at in Foreman's ``informal working definition,'' but it doesn't seem to have been seriously explored.)

One way of making sense of this is in the context of a pluralist approach to set-theoretic ontology.  This is the view that there is not a single correct mathematical universe, but a multiverse of many equally valid universes.  Not all universe-existence hypotheses have equal status.  For example, questions such as whether there are universes realizing various assignments of values to positions in the Cicho\'n diagram,\footnote{See, for example, \cite{cichon}.} or in which $\aleph_\omega$ is a J\'onsson cardinal,\footnote{See \cite{eisworth}.} are considered open questions in set theory to be settled on the basis of the existence of models of ZFC (plus conventional large cardinals).  Generic large cardinal principles may be formulated as existence principles for the multiverse.  For example, instead of asserting conjunctions ruled out by Theorem \ref{genhuge}, we could assert that there is a universe $V_1$ in which $\omega_1$ is minimally generically 2-huge, and another $V_2$ in which $\omega_2$ is minimally generically 1-huge.  A very similar approach that maintains a commitment to a single correct universe $V$ could be to formulate generic large cardinal axioms as asserting the existence of inner models of $V' \subseteq V$, along with generic objects $G$ over $V'$, generating generic embeddings $j: V' \to M \subseteq V'[G] \subseteq V$.

A potential utility to this approach is that it introduces more methods for tackling consistency problems.  A multiverse approach to generic large cardinal principles may provide a collection of well-motivated starting points for solutions to consistency questions that may not at present be answerable by other means.  The use of a conventional large cardinal assumption to prove the consistency of a theory $T$ is generally regarded as progress on (if not a complete solution to) the question of whether $T$ is consistent.  Of course, reducing the strength of the large cardinal assumption employed, or eliminating the use of large cardinals altogether, is a better result.  If $T$ can be shown \emph{equiconsistent} with ZFC plus a large cardinal, then, as seems to be universally agreed, this is the best one can say about the consistency of $T$.  In \cite{potent}, Foreman gives some examples of questions about graph theory and algebra that can be settled on the basis of generic large cardinal assumptions.  Since the consistency of many generic large cardinal assumptions, including those used in \cite{potent}, is currently not known to follow from conventional large cardinals, some nontrivial information seems to be obtained about these problems.  Though it may be preferable to solve these problems with conventional large cardinals, it is certainly of some value to find an argument from other reasonable hypotheses.  Another example whose history nicely bolsters this epistemic picture is described in \cite{fd2}.  Woodin first showed how to construct a uniform, countably complete, $\omega_1$-dense ideal on $\omega_2$, an object whose existence has many combinatorial consequences, from the assumption that both $\omega_1$ and $\omega_2$ are minimally generically almost-huge.  Foreman later showed the consistency of such an object assuming the consistency of ZFC with a huge cardinal (a type of conventional large cardinal).

It would be desirable to subsume these universe-existence hypotheses under one general principle about the nature of the set-theoretic multiverse.  Since investigations of consistency must often start from some strong assumptions, it would be valuable to have a general account of what starting assumptions are appropriate.  However, this is likely to be at least as difficult as unifying the conventional large cardinals under a single formal framework.  Thus we leave this task for future work.

\bibliographystyle{amsplain}
\bibliography{masterbib}

\end{document}